\documentclass[11pt]{article}
\usepackage[margin=1in]{geometry} 
\usepackage{amssymb,amsfonts,amsmath,bbm,mathrsfs,stmaryrd}
\usepackage{xcolor}
\usepackage{url}

\usepackage{enumerate}

\usepackage{hyperref}
\hypersetup{colorlinks,
             linkcolor=black!75!red,
             citecolor=blue,
             pdftitle={},
             pdfproducer={pdfLaTeX},
             pdfpagemode=None,
             bookmarksopen=true
             bookmarksnumbered=true}

\usepackage{tikz}
\usetikzlibrary{arrows,calc,decorations.pathreplacing,decorations.markings,intersections,shapes.geometric,through,fit,shapes.symbols,positioning,decorations.pathmorphing}

\usepackage{braket}

\usepackage[amsmath,thmmarks,hyperref]{ntheorem}
\usepackage{cleveref}

\creflabelformat{enumi}{#2(#1)#3}

\crefname{section}{Section}{Sections}
\crefformat{section}{#2Section~#1#3} 
\Crefformat{section}{#2Section~#1#3} 

\crefname{subsection}{\S}{\S\S}
\AtBeginDocument{%
  \crefformat{subsection}{#2\S#1#3}%
  \Crefformat{subsection}{#2\S#1#3}%
}

\crefname{subsubsection}{\S}{\S\S}
\AtBeginDocument{%
  \crefformat{subsubsection}{#2\S#1#3}%
  \Crefformat{subsubsection}{#2\S#1#3}%
}

\theoremstyle{plain}

\newtheorem{lemma}{Lemma}[section]
\newtheorem{proposition}[lemma]{Proposition}
\newtheorem{corollary}[lemma]{Corollary}
\newtheorem{theorem}[lemma]{Theorem}

\theoremstyle{nonumberplain}

\theoremstyle{plain}
\theorembodyfont{\upshape}
\theoremsymbol{\ensuremath{\blacklozenge}}

\newtheorem{definition}[lemma]{Definition}
\newtheorem{example}[lemma]{Example}
\newtheorem{remark}[lemma]{Remark}

\newtheorem{notation}[lemma]{Notation}

\crefname{definition}{definition}{definitions}
\crefformat{definition}{#2definition~#1#3} 
\Crefformat{definition}{#2Definition~#1#3} 

\crefname{ex}{example}{examples}
\crefformat{example}{#2example~#1#3} 
\Crefformat{example}{#2Example~#1#3} 

\crefname{remark}{remark}{remarks}
\crefformat{remark}{#2remark~#1#3} 
\Crefformat{remark}{#2Remark~#1#3} 

\crefname{convention}{convention}{conventions}
\crefformat{convention}{#2convention~#1#3} 
\Crefformat{convention}{#2Convention~#1#3} 

\crefname{notation}{notation}{notations}
\crefformat{notation}{#2notation~#1#3} 
\Crefformat{notation}{#2Notation~#1#3} 

\crefname{table}{table}{tables}
\crefformat{table}{#2table~#1#3} 
\Crefformat{table}{#2Table~#1#3}

\crefname{lemma}{lemma}{lemmas}
\crefformat{lemma}{#2lemma~#1#3} 
\Crefformat{lemma}{#2Lemma~#1#3} 

\crefname{proposition}{proposition}{propositions}
\crefformat{proposition}{#2proposition~#1#3} 
\Crefformat{proposition}{#2Proposition~#1#3} 

\crefname{corollary}{corollary}{corollaries}
\crefformat{corollary}{#2corollary~#1#3} 
\Crefformat{corollary}{#2Corollary~#1#3} 

\crefname{theorem}{theorem}{theorems}
\crefformat{theorem}{#2theorem~#1#3} 
\Crefformat{theorem}{#2Theorem~#1#3} 

\crefname{enumi}{}{}
\crefformat{enumi}{(#2#1#3)}
\Crefformat{enumi}{(#2#1#3)}

\crefname{assumption}{assumption}{Assumptions}
\crefformat{assumption}{#2assumption~#1#3} 
\Crefformat{assumption}{#2Assumption~#1#3} 

\crefname{equation}{}{}
\crefformat{equation}{(#2#1#3)} 
\Crefformat{equation}{(#2#1#3)}

\numberwithin{equation}{section}

\theoremstyle{nonumberplain}
\theoremsymbol{\ensuremath{\blacksquare}}

\newtheorem{proof}{Proof}
\newcommand\pf[1]{\newtheorem{#1}{Proof of \Cref{#1}}}

\newcommand\bA{{\mathbb A}}
\newcommand\bB{{\mathbb B}}
\newcommand\bC{{\mathbb C}}

\newcommand\bG{{\mathbb G}}
\newcommand\bH{{\mathbb H}}

\newcommand\bK{{\mathbb K}}

\newcommand\bM{{\mathbb M}}
\newcommand\bN{{\mathbb N}}

\newcommand\bP{{\mathbb P}}

\newcommand\bR{{\mathbb R}}

\newcommand\bZ{{\mathbb Z}}

\newcommand\cK{{\mathcal K}}

\newcommand\cM{{\mathcal M}}

\newcommand\cO{{\mathcal O}}

\newcommand\cR{{\mathcal R}}

\DeclareMathOperator{\id}{id}

\newcommand{\cat}[1]{\textsc{#1}}

\newcommand{\qedhere}{\mbox{}\hfill\ensuremath{\blacksquare}}

\renewcommand{\square}{\mathrel{\Box}}

\title{Normal subgroups and relative centers of linearly reductive quantum groups}
\author{Alexandru Chirvasitu}

\begin{document}

\date{}

\newcommand{\Addresses}{{
  \bigskip
  \footnotesize

  \textsc{Department of Mathematics, University at Buffalo, Buffalo,
    NY 14260-2900, USA}\par\nopagebreak \textit{E-mail address}:
  \texttt{achirvas@buffalo.edu}

}}

\maketitle

\begin{abstract}
  We prove a number of structural and representation-theoretic results on linearly reductive quantum groups, i.e. objects dual to that of cosemisimple Hopf algebras: (a) a closed normal quantum subgroup is automatically linearly reductive if its squared antipode leaves invariant each simple subcoalgebra of the underlying Hopf algebra; (b) for a normal embedding $\mathbb{H}\trianglelefteq \mathbb{G}$ there is a Clifford-style correspondence between two equivalence relations on irreducible $\mathbb{G}$- and, respectively, $\mathbb{H}$-representations; and (c) given an embedding $\mathbb{H}\le \mathbb{G}$ of linearly reductive quantum groups the Pontryagin dual of the relative center $Z(\mathbb{G})\cap \mathbb{H}$ can be described by generators and relations, with one generator $g_V$ for each irreducible $\mathbb{G}$-representation $V$ and one relation $g_U=g_Vg_W$ whenever $U$ and $V\otimes W$ are not disjoint over $\mathbb{H}$.

  This latter center-reconstruction result generalizes and recovers M\"uger's compact-group analogue and the author's quantum-group version of that earlier result by setting $\mathbb{H}=\mathbb{G}$.
\end{abstract}


\vspace{.5cm}



\section*{Introduction}

The quantum groups in the title are as in \cite[\S 1.2]{pw}: objects $\bG$ dual to corresponding Hopf algebras $\cO(\bG)$, with the latter regarded as the algebra of regular functions on (the otherwise non-existent) linear algebraic quantum group $\bG$. Borrowing standard linear-algebraic-group terminology (e.g. \cite[Chapter 1, \S 1, Definition 1.4]{fkm}), the linear reductivity condition then simply means that the Hopf algebra $\cO(\bG)$ is cosemisimple.

The unifying thread through the material below is the concept of a (closed) normal quantum subgroup. In the present non-commutative setting normality can be defined in a number of ways that are frequently equivalent \cite[Theorem 2.7]{wnorm}. We settle here on the concept introduced in \cite[\S 1.5]{pw} (and recalled in \Cref{def:norm}): a quotient Hopf algebra
\begin{equation*}
  \cO(\bG)\to \cO(\bH)
\end{equation*}
dual to a closed quantum subgroup $\bH\le \bG$ is normal if that quotient is an $\cO(\bG)$-comodule under both adjoint coactions $\cO(\bG)\to \cO(\bG)^{\otimes 2}$:
\begin{equation*}
  x\mapsto x_2\otimes S(x_1)x_3\quad\text{and}\quad x\mapsto x_1S(x_3) \otimes x_2
\end{equation*}

One piece of motivation for the material is the observation (cf. \Cref{re:clsnorm}) that classically, normal closed subgroups of linearly reductive algebraic groups are again linearly reductive. The non-commutative version of this remark, appearing as \Cref{th:isred} below, can be phrased (in somewhat weakened but briefer form) as follows.

\begin{theorem}
  A normal quantum subgroup $\bH\trianglelefteq \bG$ of a linearly reductive quantum group is again linearly reductive, provided the squared antipode of $\cO(\bH)$ leaves invariant all simple subcoalgebras of the latter.
\end{theorem}

In particular, this recovers the classical version: in that case the squared antipode is trivial.

Keeping with the theme of what is (or isn't) afforded by normality, another motivating strand is that of {\it Clifford theory} (so named for \cite{clif}, where the relevant machinery was introduced). This is a suite of results relating the irreducible representations of a (finite, compact, etc.) group and those of a normal subgroup via induction/restriction functors; the reader can find a brief illuminating summary in \cite[\S 2]{cst} (in the context of finite groups).

Hopf-algebra analogues (both purely algebraic and analytic) abound. Not coming close to doing the literature justice, we will point to a selection: \cite{with-clif-hopf,with-clif-alg,bur-clif,sch-rep,oz}, say, and the references therein. \cite[\S 5, especially Theorem 5.4]{cks} provides a version for {\it compact quantum groups} \cite{wor-cqg}, which are (dual to) cosemisimple complex Hopf $*$-algebras with positive Haar integral (the {\it CQG algebras} of \cite[Definition 2.2]{dk}); they thus fit within the confines of the present paper.

The following result paraphrases and summarizes \Cref{th.simh}, \Cref{th.simb} and \Cref{pr.2rels}. To make sense of it:
\begin{itemize}
\item In the language of \Cref{se:cliff}, the surjection $\cO(\bG)\to \cO(\bH)$ of \Cref{th:clifsum} is $H\to B$.  
\item As explained in \Cref{se.prel}, for a quantum group $\bG$ the symbol $\widehat{\bG}$ denotes its category of irreducible representations (i.e. simple right $\cO(\bG)$-comodules).
\item $\mathrm{Ind}^{\bG}_{\bH}$ and $\mathrm{Res}^{\bG}_{\bH}$ denote the induction and restriction functors respectively, as discussed in \Cref{subse:resind}.
\end{itemize}

\begin{theorem}\label{th:clifsum}
  Let $\bH\trianglelefteq\bG$ be a normal embedding of linearly reductive quantum groups, and consider the binary relation $\sim$ on $\widehat{\bG}\times \widehat{\bH}$ defined by
  \begin{equation*}
    \widehat{\bG}\ni V\sim W\in \widehat{\bH}\Leftrightarrow \mathrm{hom}_{\bH}\left(\mathrm{Res}^{\bG}_{\bH}V,W\right)\ne 0\Leftrightarrow \mathrm{hom}_{\bG}\left(V,\mathrm{Ind}^{\bG}_{\bH}W\right)\ne 0.
  \end{equation*}
  The following statements hold.
  \begin{enumerate}[(a)]
  \item The left-hand slices
    \begin{equation*}
      \mathrm{slice}_W:=\{V\in \widehat{\bG}\ |\ V\sim W\},\ W\in \widehat{\bH}
    \end{equation*}
    of $\sim$ are the classes of an equivalence relation $\sim_{\bG}$, given by
    \begin{equation*}
      V\sim_{\bG} V'\Leftrightarrow \mathrm{Res}^{\bG}_{\bH}V\text{ and }\mathrm{Res}^{\bG}_{\bH}V'\text{ have the same simple constituents}. 
    \end{equation*}
  \item The right-hand slices
    \begin{equation*}
      {}_V\mathrm{slice}:=\{W\in \widehat{\bH}\ |\ V\sim W\},\ V\in \widehat{\bG}
    \end{equation*}
    are the finite classes of an equivalence relation.
  \end{enumerate}
\end{theorem}

A third branch of the present discussion has to do with the {\it relative centers} of the title: having defined the center $Z(\bG)$ of a linearly reductive quantum group (\Cref{def:cent}), and given a closed linearly reductive quantum subgroup $\bH\le \bG$, one can then make sense of the relative center $Z(\bG,\bH)$ as the {\it intersection} $\bH\cap Z(\bG)$; see \Cref{def:relcent}.

Though not immediately obvious, it follows from \cite[\S 3]{chk} (cited more precisely in the text below) that for embeddings $\bH,\bK\le \bG$ of linearly reductive quantum groups, operations such as the intersection $\bH\cap \bK$ and the quantum subgroup $\bH\bK$ generated by the two are well defined and behave as usual when $\bK$, say, is normal (hence the relevance of normality, again).

The initial spark of motivation for \Cref{se:rel} was provided by the main result of \cite{mug} (Theorem 3.1 therein), reconstructing the center of a compact group $\bG$ as a universal grading group for the category of $\bG$-representations. This generalizes to linearly reductive {\it quantum} groups \cite[Proposition 2.9]{chi-coc}, and, as it turns out, goes through in the relative setting; per \Cref{th:caniso}:

\begin{theorem}\label{th:canisopre}
  Let $\bH\le \bG$ be an embedding of linearly reductive quantum groups, and define the relative chain group $C(\bG,\bH)$ by generators $g_V$, $V\in \widehat{\bG}$ and relations $g_U=g_Vg_W$ whenever $U$ and $V\otimes W$ have common simple constituents over $\bH$.

  Then, the map
  \begin{equation*}
    C(\bG,\bH)\ni g_V\mapsto W\in \widehat{Z(\bG,\bH)}\quad\text{where}\quad \mathrm{Res}^{\bG}_{Z(\bG,\bH)}\cong \text{sum of copies of }W
  \end{equation*}
  is a group isomorphism. 
\end{theorem}
Or, in words: mapping $g_V$ to the ``central character'' of $V$ restricted to $Z(\bG,\bH)$ gives an isomorphism $C(\bG,\bH)\cong \widehat{Z(\bG,\bH)}$. The ``plain'' (non-relative) version \cite[Proposition 2.9]{chi-coc} (and hence also its classical compact-group counterpart \cite[Theorem 3.1]{mug}) are recovered by setting $\bH=\bG$.

Although strictly speaking outside the scope of the present paper, some further remarks, suggestive of an intriguing connection to semisimple-Lie-group representation theory, will perhaps serve to further motivate the relative chain groups discussed in \Cref{th:canisopre}. 

\Cref{def:cg} was inspired by the study of plain (non-relative) chain groups of connected, semisimple Lie groups $\bG$ with finite center, studied in \cite[\S 4]{chi-cc}; specifically, the problem of whether
\begin{equation}\label{eq:hmdet}
  \mathrm{hom}_{\bH}(\sigma'',\sigma\otimes\sigma')\ne 0,\quad \sigma,\ \sigma',\ \sigma''\in \widehat{\bM}
\end{equation}
for a compact-group embedding $\bH\le \bM$ arises naturally while studying the direct-integral decomposition of a tensor product of two {\it principal-series} representations of such a Lie group $\bG$. To summarize, consider the setup of \cite{mart-dec} (to which we also refer, along with its own references, for background on the following).
\begin{itemize}
\item a connected, semisimple Lie group $\bG$ with finite center, with its {\it Iwasawa decomposition}
  \begin{equation*}
    \bG = \bK\bA\bN
  \end{equation*}
  ($\bK\le \bG$ maximal compact, $\bA$ abelian and simply-connected, $\bN$ nilpotent and simply-connected);
\item the corresponding decomposition
  \begin{equation*}
    \bP=\bM\bA\bN
  \end{equation*}
  of a minimal parabolic subgroup, with $\bM\le \bK$ commuting with $\bA$;
\item the resulting {\it principal-series} unitary representations
  \begin{equation*}
    \pi_{\sigma,\nu}:=\mathrm{Ind}_{\bP}^{\bG}(\sigma\otimes\nu\otimes\mathrm{triv}),
  \end{equation*}
  where $\sigma\in \widehat{\bM}$ and $\nu\in\widehat{\bA}$ unitary irreducible representations over those groups.
\end{itemize}
One is then interested in which $\pi_{\sigma'',\nu''}$ are {\it weakly contained} \cite[Definition F.1.1]{bdv} in tensor products $\pi_{\sigma,\nu}\otimes \pi_{\sigma',\nu'}$ (i.e. feature in a direct-integral decomposition of the latter); we write
\begin{equation*}
  \pi_{\sigma'',\nu''}\preceq \pi_{\sigma,\nu}\otimes \pi_{\sigma',\nu'}. 
\end{equation*}
It turns out that in the cases worked out in the literature there is a closed subgroup $\bH\le \bM$ that determines this weak containment via \Cref{eq:hmdet}. Examples:
\begin{itemize}
\item When the (connected, etc.) Lie group $\bG$ is {\it complex}, one can simply take $\bH=\bZ(\bG)$ (the center of $\bG$, which is always automatically contained in $\bM$). This follows, for instance, from \cite[Theorem 3.5.5]{wil-tens} in conjunction with \cite[Theorems 1 and 2]{mart-dec}. 
\item For $\bG=\mathrm{SL}(n,\bR)$, $n\ge 2$ one can again set $\bH=Z(\bG)$: \cite[\S 4]{repk} for $n=2$ and \cite[p.210, Theorem]{mart-dec} for the rest. 
\item Finally, for {\it real-rank-one} $\bG$ the main result of \cite{mart-dec}, Theorem 16 of that paper, provides such an $\bH\le \bM$ (denoted there by $\bM_0$; it is in general non-central, and in fact not even normal).
\end{itemize}
The phenomenon presumably merits some attention in its own right.

\subsection*{Acknowledgements}

This work is partially supported by NSF grant DMS-2001128

\section{Preliminaries}\label{se.prel}

Everything in sight (algebras, coalgebras, etc.) will be linear over a fixed algebraically closed field $k$. We assume some background on coalgebras and Hopf algebras, as covered by any number of good sources such as \cite{swe,abe,mont,rad}. 

\begin{notation}
  A number of notational conventions will be in place throughout. 
  \begin{itemize}
  \item $\Delta$, $\varepsilon$ and $S$ denote, respectively, coproducts, counits and antipodes. They will occasionally be decorated with letters indicating which coalgebra, Hopf algebra, etc. they are attached to; $S_{H}$, for instance, is the antipode of the Hopf algebra $H$. 
  \item We use an un-parenthesized version of {\it Heyneman-Sweedler notation} (\cite[Notation 1.4.2]{mont} or \cite[\S 2.1]{rad}):
    \begin{equation*}
      \Delta(c) = c_1\otimes c_2,\ ((\Delta\otimes\id)\circ \Delta)(c) = c_1\otimes c_2\otimes c_3 
    \end{equation*}
    and so on for coproducts and
    \begin{equation*}
      c\mapsto c_0\otimes c_1,\quad c\mapsto c_{-1}\otimes c_0
    \end{equation*}
    for right and left comodule structures respectively.   
  \item $\cO(\bG)$, $\cO(\bH)$, and so on denote Hopf algebras over a fixed algebraically closed field $k$; they are to be thought of as algebras of representative functions on linear algebraic {\it quantum groups} $\bG$, $H$, etc.
  \item An {\it embedding} $\bH\le \bG$ of quantum groups means a Hopf algebra surjection $\cO(\bG)\twoheadrightarrow \cO(\bH)$ and more generally, a morphism $\bH\to \bG$ is one of Hopf algebras in the opposite direction $\cO(\bG)\to \cO(\bH)$.
  \item Categories of (co)modules are denoted by $\cM$, decorated with the symbol depicting the (co)algebra, with the left/right position of the decoration matching the chirality of the (co)module structure. Examples: ${}_A\cM$ means left $A$-modules, $\cM^C$ denotes right $C$-comodules, etc. Comodule structures are right unless specified otherwise.     
  \item These conventions extend to {\it relative Hopf modules} (\cite[\S 8.5]{mont} or \cite[\S 9.2]{rad}): if, say, $A$ is a right comodule algebra \cite[Definition 4.1.2]{mont} over a Hopf algebra $H$ with structure
    \begin{equation*}
      A\ni a\mapsto a_0\otimes a_1\in A\otimes H
    \end{equation*}
    then $\cM^H_A$ denotes the category of right $A$-modules internal to $\cM^H$; that is, right $A$-modules $M$ that are also right $H$-comodules via
    \begin{equation*}
      m\mapsto m_0\otimes m_1
    \end{equation*}
    such that
    \begin{equation*}
      (ma)_0\otimes (ma)_1 = m_0a_0\otimes m_1a_1.
    \end{equation*}
    There are analogues $\cM_H^C$, say, for right $H$-module coalgebras $C$, left- or half-left-handed versions thereof, and so on. 
  \item An additional `$f$' adornment on one of the above-mentioned categories means {\it finite-dimensional} (co)modules: $\cM^C_f$ is the category of finite-dimensional right $C$-comodules, for instance.
  \item Reprising a convention common in the operator-algebra literature (e.g. \cite[\S 2.3.2, \S 18.1.1]{dixc}), $\widehat{C}$ denotes the isomorphism classes of simple and hence finite-dimensional \cite[Theorem 5.1.1]{mont} (right, unless specified otherwise) $C$-comodules and $\widehat{\bG}=\widehat{\cO(\bG)}$.

    The purely-algebraic and operator-algebraic notations converge when $\bG$ is compact and $\cO(\bG)$ denotes the Hopf algebra of representative functions on $\bG$: $\widehat{\bG}$ as defined above can then be identified with the set of isomorphism classes of irreducible unitary $\bG$-representations.
  \item In the same spirit, it will also occasionally be convenient to write
    \begin{equation*}
      \mathrm{Rep}(\bG):=\cM^{\cO(\bG)}. 
    \end{equation*}
  \end{itemize}
\end{notation}

The linear algebraic quantum groups $\bG$ in the sequel will frequently be {\it linearly reductive}, in the sense that the Hopf algebra $\cO(\bG)$ is {\it cosemisimple} \cite[\S 2.4]{mont}: $\mathrm{Rep}(\bG)$ is a semisimple category, i.e. every comodule is a direct sum of simple subcomodules. Equivalently (\cite[Definition 2.4.1]{mont}), $\cO(\bG)$ is a direct sum of simple subcoalgebras.

Cosemisimple Hopf algebras $H$ are equipped with unique unital {\it integrals} $\int:H\to k$ \cite[Theorem 2.4.6]{mont} and hence have bijective antipodes (by \cite[Corollary 5.4.6]{dnr}, say); more is true, though. Still assuming $H$ cosemisimple, for a simple comodule $V\in \widehat{H}$ the canonical coalgebra morphism
\begin{equation*}
  \mathrm{End}(V)^*\cong V^*\otimes V\to H
\end{equation*}
(conceptually dual to the analogous map $A\to \mathrm{End}(V)$ giving $V$ a module structure over an algebra $A$) is one-to-one and gives the direct-sum decomposition
\begin{equation}\label{eq:pwdec}
  H=\bigoplus_{V\in \widehat{H}} (V^*\otimes V) = \bigoplus_{V\in \widehat{H}} C_V
\end{equation}
into simple subcoalgebras $C_V:=V^*\otimes V$ (the {\it Peter-Weyl} decomposition, in compact-group parlance: \cite[Definition 2.2]{dk}, \cite[Theorem 27.40]{hr2}, etc.) that makes $H$ cosemisimple to begin with. With this in place, not only is the antipode $S:=S_H$ bijective but in fact its square leaves every $C_V$, $V\in \widehat{H}$ invariant and acts as an automorphism thereon \cite[Theorem 7.3.7]{dnr}. 

We refer to $C_V=V^*\otimes V$ as the {\it coefficient coalgebra} of the simple $H$-comodule $V$. This is the coalgebra associated to $V$ in \cite[Proposition 2.5.3]{dnr}, and is the smallest subcoalgebra $C\le H$ for which the comodule structure
\begin{equation*}
  V\to V\otimes H
\end{equation*}
factors through $V\otimes C$. 

\subsection{Restriction, induction and the like}\label{subse:resind}

Given a coalgebra morphism $C\to D$, the {\it cotensor product} (\cite[Definition 8.4.2]{mont} or \cite[\S 10]{bw}) $-\square_D C$ is right adjoint to the natural ``scalar corestriction'' functor $\cM^C\to \cM^D$:
\begin{equation}\label{eq:cotens}
  \begin{tikzpicture}[auto,baseline=(current  bounding  box.center)]
    \path[anchor=base] 
    (0,0) node (l) {$\cM^C$}
    +(2,0) node (m) {$\bot$}
    +(4,0) node (r) {$\cM^D$}
    ;
    \draw[->] (l) to[bend left=16] node[pos=.5,auto] {$\scriptstyle \text{cores}$} (r);
    \draw[->] (r) to[bend left=16] node[pos=.5,auto] {$\scriptstyle -\square_DC$} (l);
  \end{tikzpicture}
\end{equation}
the central symbol indicating that the top functor is the left adjoint. When $\bH\le \bG$ is, say, an inclusion of compact groups and $C\to D$ the corresponding surjection $\cO(\bG)\to \cO(\bH)$ of algebras of representative functions, the cotensor functor
\begin{equation*}
  -\square_{\cO(\bH)}\cO(\bG):\mathrm{Rep}(\bH)\to \mathrm{Rep}(\bG)
\end{equation*}
is naturally isomorphic with the usual {\it induction} $\mathrm{Ind}_H^\bG$ \cite[p.82]{rob}. For that reason we repurpose this same notation for the general setting of quantum-group inclusions, writing
\begin{equation*}
  \mathrm{Ind}_{\bH}^\bG:= -\square_{\cO(\bH)}\cO(\bG):\mathrm{Rep}(\bH) \to \mathrm{Rep}(\bG)
\end{equation*}
for any quantum-group inclusion $\bH\le \bG$; for consistency, we also occasionally also denote the rightward functor in \Cref{eq:cotens} by
\begin{equation*}
  \mathrm{Res}_{\bH}^\bG:\mathrm{Rep}(\bG) \to \mathrm{Rep}(\bH).
\end{equation*}

\section{Normal subgroups and automatic reductivity}\label{se:norm}

Consider a quantum group embedding $\bH\le \bG$, expressed as a surjective Hopf-algebra morphism $\pi:\cO(\bG)\to \cO(\bH)$. As is customary in the literature on quantum homogeneous spaces (e.g. \cite[proof of Theorem 2.7]{wnorm}), we write
\begin{align*}
  \cO(\bG/\bH) &:= \{x\in\cO(\bG)\ |\ (\id\otimes\pi)\Delta(x) = x\otimes 1\}\\
  \cO(\bH\backslash \bG) &:= \{x\in\cO(\bG)\ |\ (\pi\otimes\id)\Delta(x) = 1\otimes x\}.\\
\end{align*}

According to \cite[Definition 1.1.5]{ad} a quantum subgroup $\bH\le \bG$ would be termed {\it normal} provided the two quantum homogeneous spaces $\cO(\bG/\bH)$ and $\cO(\bH\backslash \bG)$ coincide. This will not quite do for our purposes (see \Cref{ex:borel}), so instead we follow \cite[\S 1.5]{pw} (also, say, \cite[Definition 2.6]{wnorm}, relying on the same source) in the following

\begin{definition}\label{def:norm}
  The quantum subgroup $\bH\le \bG$ cast as the surjection $\pi:\cO(\bG)\to \cO(\bH)$
  \begin{itemize}
  \item {\it left-normal} if $\pi$ is a morphism of left $\cO(\bG)$-comodules under the left adjoint coaction
    \begin{equation*}
      \mathrm{ad}_{l}:=\mathrm{ad}_{l,\bG}:x\mapsto x_1S(x_3)\otimes x_2.
    \end{equation*}
  \item {\it right-normal} if similarly, $\pi$ is a morphism of right $\cO(\bG)$-comodules under the right adjoint coaction
    \begin{equation}\label{eq:radj}
      \mathrm{ad}_{r}:=\mathrm{ad}_{r,\bG}:x\mapsto x_2\otimes S(x_1)x_3.
    \end{equation}
  \item {\it normal} if it is both left- and right-normal. 
  \end{itemize}
\end{definition}

The following result is essentially a tautology in the framework of \cite[\S 1.2]{chk}, but only because in that paper the definition of a normal quantum subgroup is more restrictive (see \cite[Definition 1.2.3]{chk}, which makes an additional (co)flatness requirement).

\begin{theorem}\label{th:isred}
  Let $\bH\le \bG$ be a left- or right-normal quantum subgroup of a linearly reductive group such that $S^2$ leaves every simple subcoalgebra of $\cO(\bH)$.

  $\bH$ is then linearly reductive and normal.
\end{theorem}

\begin{remark}
  The condition that $S^2$ leave invariant the simple subcoalgebras is certainly necessary for cosemisimplicity \cite[Theorem 7.3.7]{dnr}, but I do not know if it is redundant as a hypothesis in the context of \Cref{th:isred}. 
\end{remark}

In particular, the squared-antipode condition of \Cref{th:isred} is automatic when $S^2=\id$ (i.e. when $\cO(\bG)$, or $\bG$, is {\it involutory} or {\it involutive} \cite[Definition 7.1.12]{rad}). We thus have

\begin{corollary}
  Left- or right-normal quantum subgroups of involutive linearly reductive quantum groups are normal and linearly reductive.  \qedhere
\end{corollary}

The proof of \Cref{th:isred} requires some preparation. First, a simple remark for future reference.

\begin{lemma}\label{le:bijant}
  Let $\pi:H\to K$ be a surjective morphism of Hopf algebras with $H$ cosemisimple. $K$ then has bijective antipode, and hence $\pi$ intertwines antipode inverses. 
\end{lemma}
\begin{proof}
  That a morphism of bialgebras intertwines antipodes or antipode inverses as soon as these exist is well known, so we focus on the claim that $S_{K}$ is bijective.

  By the very definition of cosemisimplicity $H$ is the direct sum of its simple (hence finite-dimensional \cite[Theorem 5.1.1]{mont}) subcoalgebras $C_i\le H$. The assumption is that $\pi$ is a morphism of Hopf algebras, so the antipode $S:=S_H$ restricts to maps
  \begin{equation}\label{eq:fsts}
    S:\ker(\pi|_{C_i})\to \ker(\pi|_{S(C_i)}),
  \end{equation}
  injective because $S$ is bijective. On the other hand though, for cosemisimple Hopf algebras the squared antipode leaves every subcoalgebra invariant \cite[Theorem 7.3.7]{dnr}, so
  \begin{equation*}
    S^2:\ker(\pi|_{C_i})\to \ker(\pi|_{S^2(C_i)}) = \ker(\pi|_{C_i}),
  \end{equation*}
  being a one-to-one endomorphism of a finite-dimensional vector space, must be bijective. Since that map decomposes as \Cref{eq:fsts} followed by its (similarly one-to-one) analogue defined on $S(C_i)$, \Cref{eq:fsts} itself must be bijective, and hence the inverse antipode $S^{-1}$ leaves $\ker(\pi)$ invariant. This, in essence, was the claim.
\end{proof}

The conclusion of \Cref{le:bijant} is by no means true of arbitrary bijective-antipode Hopf algebras $H$:

\begin{example}\label{ex:nobij}
  \cite[Theorem 3.2]{schau-ff} gives an example of a Hopf algebra $H$ with bijective antipode and a Hopf ideal $I\trianglelefteq H$ that is not invariant under the inverse antipode. In other words, even though $H$ has bijective antipode, the quotient Hopf algebra $H\to H/I$ does not. 
\end{example}

\pf{th:isred}
\begin{th:isred}
  The proof proceeds gradually.
  
  \vspace{.5cm}

  {\bf Step 1: normality.} According to \Cref{le:bijant} the antipode $S:=S_{\cO(\bG)}$ and its inverse both leave the kernel $\cK$ of the surjection
  \begin{equation*}
    \pi:\cO(\bG)\to \cO(\bH)
  \end{equation*}
  invariant, so $S(\cK)=\cK$. The fact that left- and right-normality are equivalent now follows from \cite[Proposition 1.5.1]{pw}. 
  
  \vspace{.5cm}

  {\bf Step 2: The homogeneous spaces $\bG/\bH$ and $\bH\backslash \bG$ coincide.} This means that
  \begin{equation}\label{eq:homogs}
    \cO(\bH\backslash \bG) = \cO(\bG/\bH) =: A,
  \end{equation}
  and follows from \cite[Lemma 1.1.7]{ad}.  

  \vspace{.5cm}

  {\bf Step 3: Reduction to trivial $\bG/\bH$.} The subspace $A\le \cO(\bG)$ of \Cref{eq:homogs} is in fact a Hopf subalgebra \cite[Lemma 1.1.4]{ad}. $A$ is also invariant under the right adjoint action
  \begin{equation*}
    \cO(\bG)\otimes \cO(\bG) \ni x\otimes y\mapsto S(y_1) x y_2\in \cO(\bG)
  \end{equation*}
  (\cite[Lemma 1.20]{chk}), so by \cite[Lemma 1.1.11]{ad} the left ideal
  \begin{equation*}
    \cO(\bG)A^-\le \cO(\bG)\text{ where }A^-:=\mathrm{ker}(\varepsilon|_A)
  \end{equation*}
  is bilateral. The quotient $\cO(\bG)/\cO(\bG)A^-$ must then be a {\it cosemisimple} quotient Hopf algebra \cite[Theorem 2.5]{chi-ff} $\cO(\bG)\to \cO(\bK)$, and we have an exact sequence
  \begin{equation*}
    \begin{tikzpicture}[auto,baseline=(current  bounding  box.center)]
      \path[anchor=base] 
      (0,0) node (ll) {$k$}
      +(1.5,.5) node (l) {$\cO(\bG/\bK)$}
      +(3.5,.5) node (m) {$\cO(\bG)$}
      +(5.5,.5) node (r) {$\cO(\bK)$}
      +(7,0) node (rr) {$k$}
      +(1.5,0) node () {$\parallel$}
      +(1.5,-.5) node () {$\cO(\bG/\bH)$}
      ;
      \draw[->] (ll) to[bend left=6] node[pos=.5,auto] {$\scriptstyle $} (l);
      \draw[->] (l) to[bend left=6] node[pos=.5,auto] {$\scriptstyle $} (m);
      \draw[->] (m) to[bend left=6] node[pos=.5,auto] {$\scriptstyle $} (r);
      \draw[->] (r) to[bend left=6] node[pos=.5,auto] {$\scriptstyle $} (rr);
    \end{tikzpicture}
  \end{equation*}
  of quantum groups in the sense of \cite[\S 1.2]{ad}, with everything in sight cosemisimple. Since furthermore $A^-$ is annihilated by the original surjection $\cO(\bG)\to \cO(\bH)$, $\bH$ can be thought of as a quantum subgroup of $\bK$ (rather than $\bG$):
  \begin{equation*}
    \cO(\bK)\to \cO(\bH).
  \end{equation*}
  I now claim that the corresponding homogeneous space is trivial:
  \begin{equation}\label{eq:khhk}
    \cO(\bK/\bH) = \cO(\bH\backslash \bK) = k.
  \end{equation}    
  To see this, consider a simple representation $V\in\widehat{\bK}$ that contains invariant vectors over $\bH$. Because $\cO(\bK)$ is cosemisimple, $V$ is a subcomodule (rather than just a subquotient) of a simple comodule $W\in \widehat{\bG}$, and it follows that
  \begin{equation*}
    W|_{\bH}\ge V|_{\bH}
  \end{equation*}
  contains invariant vectors. The fact that \Cref{eq:homogs} is a Hopf subalgebra means that it is precisely
  \begin{equation*}
    \bigoplus_{U}C_U,\ U\in\widehat{\bG}\text{ and }U|_{\bH}\text{ has invariant vectors},
  \end{equation*}
  so $C_W\le A$ and the restriction $W|_{\bK}$ decomposes completely as a sum of copies of the trivial comodule $k$. But then $V\le W|_{\bK}$ itself must be trivial, proving the claim \Cref{eq:khhk}. Now simply switch the notation back to $\bG:=\bK$ to conclude Step 3:
  \begin{equation}\label{eq:nohomog}
    \cO(\bG/\bH) = \cO(\bH\backslash \bG) = k.
  \end{equation}
  This latter condition simply means that for an $\cO(\bG)$-comodule $V$ its $\bG$- and $\bH$-invariants coincide:
  \begin{equation*}
    \mathrm{hom}_{\bG}(k,V) = \mathrm{hom}_{\bH}(k,V). 
  \end{equation*}
  Equivalently, since
  \begin{equation*}
    \mathrm{hom}_{\bG}(V,W) = \mathrm{hom}_{\bG}(k,W\otimes V^*), 
  \end{equation*}
  this simply means that the restriction functor
  \begin{equation}\label{eq:res}
    \mathrm{Rep}(\bG)\ni V\mapsto V|_{\bH}\in \mathrm{Rep}(\bH)
  \end{equation}
  is full (for both left and right comodules, but here we focus on the latter).  
  
  \vspace{.5cm}

  {\bf Step 4: Wrapping up.} Because the restriction functor \Cref{eq:res} is full, simple, non-isomorphic $\bG$-representations that remain simple over $\bH$ also remain non-isomorphic. 

  Now, assuming $\bH\le \bG$ is not an isomorphism (or there would be nothing to prove), some irreducible $V\in\widehat{\bG}$ must become reducible over $\bH$. There are two possibilities to consider:
  \begin{enumerate}[(a)]
  \item\label{item:1} All simple subquotients of the reducible representation $V|_{\bH}$ are isomorphic. We then have (in $\mathrm{Rep}(\bH)$) a surjection of $V$ onto a simple quotient thereof, which then embeds into $V$ again. All in all this gives a non-scalar endomorphism of $V$ over $\bH$, contradicting the fullness of the restriction functor \Cref{eq:res}. 
  \item\label{item:2} $V$ acquires at least two non-isomorphic simple subquotients $V_i$, $i=1,2$ over $\bH$. Then, the image of the coefficient coalgebra $C_V=V^*\otimes V$ of \Cref{eq:pwdec} through $\pi:\cO(\bG)\to \cO(\bH)$ will contain both
    \begin{equation*}
      C_{V_i} = V_i^*\otimes V_i\le \cO(\bH),\ i=1,2
    \end{equation*}
    as (simple) subcoalgebras. 
    
    The requirement that $S^2(C_{V_i})=C_{V_i}$ means that the simple comodules $V_i$ are isomorphic to their respective double duals $V_i^{**}$ (as $\cO(\bH)$-comodules, not just vector spaces). But then
    \begin{equation*}
      C_{V_i} = V_i^*\otimes V_i\cong V_i^*\otimes V_i^{**}
    \end{equation*}
    contains an $\bH$-invariant vector, namely the image of the {\it coevaluation} \cite[Definition 9.3.1]{maj-fnd}
    \begin{equation*}
      \mathrm{coev}_{V_i^*}:k\to V_i^*\otimes V_i^{**}. 
    \end{equation*}
    It follows that the space of $\bH$-invariants of the $\cO(\bG)$-comodule $\pi(C_V)$ is at least 2-dimensional, whereas that of $\bG$-invariants is at most 1-dimensional (because the same holds true of $C_V=V^*\otimes V$). This contradicts the fullness of \Cref{eq:res} and hence our assumption that $\bH\le \bG$ is {\it not} an isomorphism.
  \end{enumerate}
  The proof of the theorem is now complete. 
\end{th:isred}

\begin{remark}
  Left and right normality are proven equivalent to an alternative notion (\cite[Definition 2.3]{wnorm}) in \cite[Theorem 2.7]{wnorm} in the context of {\it CQG algebras}, i.e. complex cosemisimple Hopf $*$-algebras with positive unital integral (this characterization is equivalent to \cite[Definition 2.2]{dk}).

  The substance of \Cref{th:isred}, however, is the cosemisimplicity claim; this is of no concern in the CQG-algebra case, as a Hopf $*$-algebra that is a quotient of a CQG algebra is automatically again CQG (as follows, for instance, from \cite[Proposition 2.4]{dk}), and hence cosemisimple.
\end{remark}

\begin{example}\label{ex:borel}
  The weaker requirement that $\cO(\bG/\bH)=\cO(\bH\backslash \bG)$ for normality would render \Cref{th:isred} false.

  Let $\bG$ be a semisimple complex algebraic group and $\bB\le \bG$ a {\it Borel subgroup} \cite[Part II, \S 1.8]{jntz}. The restriction functor
  \begin{equation*}
    \mathrm{Res}:\mathrm{Rep}(\bG)\to \mathrm{Rep}(\bB) 
  \end{equation*}
  is full \cite[Part II, Corollary 4.7]{jntz}, so in particular
  \begin{equation*}
    \cO(\bG/\bB) = \mathrm{hom}_\bB(\mathrm{triv},\cO(\bG)) = \mathrm{hom}_\bG(\mathrm{triv},\cO(\bG)) = \bC
  \end{equation*}
  and similarly for $\cO(\bB\backslash \bG)$. This means that $\cO(\bG/\bB)=\cO(\bB\backslash \bG)$, but $\bB$ is nevertheless not reductive.
\end{example}

\begin{remark}\label{re:clsnorm}
  The classical (as opposed to quantum) analogue of \Cref{th:isred} admits an alternative, more direct proof relying on the structure of reductive groups:
  \begin{itemize}
  \item In characteristic zero linear reductivity is equivalent (by \cite[p.88 (2)]{nag}, for instance) to plain reductivity \cite[\S 11.21]{brl}, i.e. the condition that the {\it unipotent radical} $\cR_u(\bG)$ of $\bG$ (the largest normal connected unipotent subgroup) be trivial.

    Assuming $\bG$ is reductive, for any normal $\bK\trianglelefteq \bG$ the corresponding unipotent radical $\cR_u(\bK)$ is characteristic in $N$ and hence normal in $\bG$, meaning that
    \begin{equation*}
      \cR_u(\bK)\le \cR_u(\bG)=\{1\}
    \end{equation*}
    and hence $N$ is again reductive (so linearly reductive, in characteristic zero). 
    
  \item On the other hand, in positive characteristic $p$ \cite[p.88 (1)]{nag} says that the linearly reductive groups $\bG$ are precisely those fitting into an exact sequence
    \begin{equation*}
      \{1\}\to \bK\to \bG\to \bG/\bK\to \{1\}
    \end{equation*}
    with $\bK$ a closed subgroup of a torus and $\bG/\bK$ finite of order coprime to $p$. Clearly then, normal subgroups of $\bG$ have the same structure.
  \end{itemize}
\end{remark}

\section{Clifford theory}\label{se:cliff}

We work with an exact sequence \Cref{eq:seq}
\begin{equation}\label{eq:seq}
  k\to A\to H\to B\to k
\end{equation}
of cosemisimple Hopf algebras in the sense of \cite[p. 23]{ad}. Note that we additionally know that $H$ is left and right coflat over $B$ (simply because the latter is cosemisimple) and left and right faithfully flat over $A$ (by \cite[Theorem 2.1]{chi-ff}).

We will make frequent use of \cite[Theorem 1]{tak-ff}, to the effect that
  \begin{equation}\label{eq:equiv-ahb}
  \begin{tikzpicture}[auto,baseline=(current  bounding  box.center)]
    \path[anchor=base] (0,0) node (1) {$\cM_A^H$} +(5,0) node (2) {$\cM^B$};
    \draw[->] (1) to[bend left=10] node [pos=.5,auto] {$\scriptstyle M\mapsto M/MA^-$} (2);
    \draw[->] (2) to[bend left=10] node [pos=.5,auto] {$\scriptstyle  N\otimes A\mapsfrom N$} (1);    
  \end{tikzpicture}
\end{equation}
is an equivalence, where the $-$ superscript denotes kernels of
counits.

Upon identifying $\cM^B$ with $\cM^H_A$ via \Cref{eq:equiv-ahb}, the adjunction
\begin{equation}\label{eq:adj}
  \begin{tikzpicture}[auto,baseline=(current  bounding  box.center)]
    \path[anchor=base] (0,0) node (1) {$\cM^H$} +(5,0) node (2) {$\cM^B$};
    \draw[->] (1) to[bend left=10] node [pos=.5,auto] {$\scriptstyle \mathrm{corestrict}$} (2);
    \draw[->] (2) to[bend left=10] node [pos=.5,auto] {$\scriptstyle  -\square_B H$} (1);    
  \end{tikzpicture}
\end{equation}

becomes 

\begin{equation}\label{eq:adj-alt}
  \begin{tikzpicture}[auto,baseline=(current  bounding  box.center)]
    \path[anchor=base] (0,0) node (1) {$\cM^H$} +(5,0) node (2) {$\cM^H_A$.};
    \draw[->] (1) to[bend left=10] node [pos=.5,auto] {$\scriptstyle -\otimes A$} (2);
    \draw[->] (2) to[bend left=10] node [pos=.5,auto] {$\scriptstyle \mathrm{forget}$} (1);    
  \end{tikzpicture}
\end{equation}

We will freely switch points of view between the two perspectives provided by \Cref{eq:adj,eq:adj-alt}. Consider the following binary relation $\sim_B$ on $\widehat{B}$.

\begin{definition}\label{def.simb}
  For $V,W\in \widehat{B}$, $V\sim_B W$ provided there is a simple $H$-comodule $U$ such that $V$ and $W$ are both constituents of the corestriction of $U$ to $B$.
\end{definition}

Similarly, we will study the following relation on $\widehat{H}$:

\begin{definition}\label{def.simh}
  For $V,W\in \widehat{H}$ we set $V\sim_H W$ provided $\hom^B(V,W)\ne 0$.
\end{definition}

\begin{remark}
  In other words, $\sim_H$ signifies the fact that the corestrictions of $V$ and $W$ to $\cM^B$ have common simple constituents.
\end{remark}

Our first observation is that $\sim_H$ is an equivalence relation, and provides an alternate characterization for it.

\begin{theorem}\label{th.simh} $\sim_H$ is an equivalence relation on $\widehat{H}$, and moreover, for $V,W\in\widehat{H}$ the following conditions are equivalent
  \begin{enumerate} \renewcommand{\labelenumi}{(\arabic{enumi})}
    \item $V\sim_H W$;
    \item as $B$-comodules, $V$ and $W$ have the same simple constituents;
    \item $V$ embeds into $W\otimes A\in \cM^H$.
  \end{enumerate}
\end{theorem}
\begin{proof} Note first that (2) clearly defines an equivalence relation on $\widehat{H}$, so the first statement of the theorem will be a consequence of
  \begin{equation*} (1) \Leftrightarrow (2) \Leftrightarrow (3).
  \end{equation*} We prove the latter result in stages.

  \vspace{.5cm}

  {\bf $(1)\Leftrightarrow (3)$.}  By definition, $V\sim_H W$ if and only if
  \begin{equation*} \hom^B(V,W)\ne 0.
  \end{equation*} Via \Cref{eq:equiv-ahb} and the hom-tensor adjunction \Cref{eq:adj-alt}, this hom space can be identified with
  \begin{equation}\label{eq:vwa} \hom^H_A(V\otimes A,W\otimes A) \cong \hom^H(V,W\otimes A).
  \end{equation} The simplicity of $V\in \widehat{H}$ now implies that every non-zero element of the right hand side of \Cref{eq:vwa} is an embedding, hence finishing the proof of the equivalence of (1) and (3).

  \vspace{.5cm}

  {\bf $(1)\Leftrightarrow (2)$.} Let us denote by $\mathrm{const}(\bullet)$ the set of simple constituents of a $B$-comodule $\bullet$.

  By definition $V\sim_HW$ means that {\it some} of the simple constituents of $V$ and $W$ as objects in $\cM^B$ coincide, so (2) is clearly stronger than (1). Conversely, note that by the equivalence (1) $\Rightarrow$ (3) proven above, whenever $V\sim_HW$ we have
  \begin{equation}\label{eq:const} \mathrm{const}(V)\subseteq \mathrm{const}(W\otimes A),
  \end{equation} where the respective objects are regarded as $B$-comodules via the corestriction functor $\cM^H\to \cM^B$.

  In turn however, given that $A\in \cM^H$ breaks up as a sum of copies of $k$ in $\cM^B$ (because of the exactness of \Cref{eq:seq}), the right hand side of \Cref{eq:const} is simply $\mathrm{const}(W)$. All in all, we have
  \begin{equation*} V\sim_HW\Rightarrow \mathrm{const}(V)\subseteq \mathrm{const}(W).
  \end{equation*} This together with the symmetry of $\sim_H$ (obvious by definition from the semisimplicity of $\cM^B$) finishes the proof of (1) $\Rightarrow$ (2) and of the theorem.
\end{proof}

\begin{theorem}\label{th.simb}
  $\sim_B$ is an equivalence relation on $\widehat{B}$ with finite classes.
\end{theorem}
\begin{proof}
  We know from \Cref{th.simh} above that as $U$ ranges over $\widehat{H}$, the sets $\mathrm{const}(U)$ of constituents of $U\in \cM^B$ partition $\widehat{B}$, thus defining an equivalence relation on the latter set.

  The definition of $\sim_B$ ensures that $V\sim_B W$ if and only if $V$ and $W$ fall in the same set $\mathrm{const}(U)$, and hence $\sim_B$ coincides with the equivalence relation from the previous paragraph.

  Finally, the statement on finiteness of classes is implicit in their description given above: an equivalence class is the set of simple constituents of a simple $H$-comodule $U$ viewed as a $B$-comodule, and it must be finite because $\mathrm{dim}(U)$ is.
\end{proof}

\Cref{th.simh,th.simb} establish a connection between the equivalence relations $\sim_H$ and $\sim_B$ on $\widehat{H}$ and $\widehat{B}$ respectively. We record it below.

Before getting to the statement, recall the notation $\mathrm{const}(\bullet)\subseteq \widehat{B}$ for the set of simple summands of an object $\bullet\in \cM^B$. With that in mind, we have the following immediate consequence of \Cref{th.simh,th.simb}.

\begin{proposition}\label{pr.2rels}
  The range of the map
  \begin{equation*}
    \widehat{H}\to \text{ finite subsets of }\widehat{B}
  \end{equation*}
  sending $V\in\widehat{H}$ to $\mathrm{const}(V)$ consists of the equivalence classes of $\sim_B$, and its fibers are the classes of $\sim_H$. \qedhere
\end{proposition}

\section{Relative chain groups and centers}\label{se:rel}

\begin{definition}\label{def:cg}
  Let $\bH\le \bG$ be an inclusion of linearly reductive quantum groups. The {\it (relative) chain group} $C(\bG,\bH)$ is defined by
  \begin{itemize}
  \item generators $g_V$ for simple comodules $V\in \widehat{\bG}$;
  \item relations
    \begin{equation}\label{eq:crel}
      \mathrm{hom}_{\bH}(U,V\otimes W)\ne 0\Rightarrow g_U = g_V g_W;
    \end{equation}
    that is, one such relation whenever the restrictions of $U$ and $V\otimes W$ to $\bH$ have non-trivial common summands (i.e. $U$ and $V\otimes W$ are not {\it disjoint} over $\bH$).
  \end{itemize}
  We write $C(\bG):=C(\bG,\bG)$.
\end{definition}

\begin{remark}\label{re:chains}
  For chained inclusions $\bK\le \bH\le \bG$ we have a map $C(\bG,\bK)\to C(\bG,\bH)$ sending the class of $V\in\widehat{\bG}$ in the domain to the class of the selfsame $V$ in the codomain. This is easily seen to be well-defined and a group morphism.
\end{remark}

Recall \cite[Definition 2.10]{chi-coc}. 

\begin{definition}\label{def:cent}
  Let $\bG$ be a linearly reductive quantum group. Its {\it center} $Z(\bG)\le \bG$ is the quantum subgroup dual to the largest Hopf algebra quotient
  \begin{equation*}
    \pi:\cO(\bG)\to \cO(Z(\bG))
  \end{equation*}
  that is central in the sense of \cite[Definition 2.1]{chi-coc}:
  \begin{equation*}
    \pi(x_1)\otimes x_2 = \pi(x_2)\otimes x_1\in \cO(Z(\bG))\otimes \cO(\bG),\ \forall x\in \cO(\bG). 
  \end{equation*}
\end{definition}

The relative version of this construction, alluded to in the title, is as follows.

\begin{definition}\label{def:relcent}
  Let $\bH\le \bG$ be an embedding of linearly reductive quantum groups. The corresponding {\it relative center} $Z(\bG,\bH)$ is the {\it intersection} $Z(\bG)\cap \bH$ denoted by $Z(\bG)\wedge \bH$ in \cite[Definition 1.15]{chk}.

  This is a quantum subgroup of both $\bH$ and $Z(\bG)$ (and hence also of $\bG$), and is automatically linearly reductive by \cite[Proposition 3.1]{chk}.
\end{definition}

Each irreducible $\bG$-representation breaks up as a sum of mutually isomorphic (one-dimensional) representations over the center $Z(\bG)$, and hence gets assigned an element of $\widehat{Z(\bG)}$: its {\it central character}. Two such irreducible representations that are not disjoint over $\bH$ must have corresponding central characters agreeing on
\begin{equation*}
  Z(\bG,\bH):=Z(\bG)\cap \bH
\end{equation*}
(the {\it relative center} associated to the inclusion $\bH\le \bG$), so we have a canonical morphism
\begin{equation}\label{eq:can}
  \cat{can}:C(\bG,\bH)\to \widehat{Z(\bG,\bH)}
\end{equation}

\begin{theorem}\label{th:caniso}
  For any embedding $\bH\le \bG$ of linearly reductive quantum groups \Cref{eq:can} is an isomorphism.
\end{theorem}
\begin{proof}
  Consider the commutative diagram
  \begin{equation}\label{eq:cans}
    \begin{tikzpicture}[auto,baseline=(current  bounding  box.center)]
      \path[anchor=base] 
      (0,0) node (l) {$C(\bG)$}
      +(3,.5) node (u) {$C(\bG,\bH)$}
      +(3,-.5) node (d) {$\widehat{Z(\bG)}$}
      +(6,0) node (r) {$\widehat{Z(\bG,\bH)}$}
      ;
      \draw[->] (l) to[bend left=6] node[pos=.5,auto] {$\scriptstyle $} (u);
      \draw[->] (u) to[bend left=6] node[pos=.5,auto] {$\scriptstyle \cat{can}$} (r);
      \draw[->] (l) to[bend right=6] node[pos=.5,auto,swap] {$\scriptstyle \cat{can}$} node[pos=.5,auto] {$\scriptstyle \cong$} (d);
      \draw[->] (d) to[bend right=6] node[pos=.5,auto,swap] {$\scriptstyle $} (r);
    \end{tikzpicture}
  \end{equation}
  where
  \begin{itemize}
  \item the upper left-hand morphism is an instance of the maps noted in \Cref{re:chains};
  \item the bottom right-hand map is the (plain) group surjection dual to the quantum-group inclusion $Z(\bG)\cap \bH\le Z(\bG)$;
  \item and the fact that the bottom left-hand map is an isomorphism is a paraphrase of \cite[Proposition 2.9]{chi-coc} in conjunction with \cite[Definition 2.10]{chi-coc}. 
  \end{itemize}
  The surjectivity of the bottom composition entails that of \Cref{eq:can}, so it remains to show that the latter is one-to-one.

  Let $V\in\widehat{\bG}$ be a simple comodule where $Z(\bG,\bH)$ operates with trivial character, i.e. one whose class in $C(\bG,\bH)$ is annihilated by \Cref{eq:can}. We can then form the quantum subgroup
  \begin{equation*}
    Z(\bG)\bH:=Z(\bG)\vee \bH\le \bG
  \end{equation*}
  generated by $Z(\bG)$ and $\bH$ as in \cite[Definition 1.15]{chk} (the `$\vee$' notation is used there; we suppress the symbol here for brevity), which then satisfies, according to \cite[Theorem 3.4]{chk}, a quantum-flavored isomorphism theorem:
  \begin{equation*}
    \bH/Z(\bG,\bH) \stackrel{\cong}{\longrightarrow} Z(\bG)\bH/Z(\bG)
  \end{equation*}
  via the canonical map induced from $\bH\to Z(\bG)\bH$. Since $V$ (or rather its restriction $V|_{\bH}$) is a representation of the former group because $Z(\bG,\bH)$ operates trivially, it lifts to a $Z(\bG)\bH$-representation with $Z(\bG)$ acting trivially. In summary:
  \begin{equation*}
    \text{The restriction }V|_{\bH}\text{ extends to a }Z(\bG)\bH\text{-representation $W$ with trivial }Z(\bG)\text{-action}.
  \end{equation*}
  But then the induced representation $\mathrm{Ind}_{Z(\bG)\bH}^\bG W$ again has trivial central character, and hence so do all of its simple summands $V_1$. The adjunction \Cref{eq:cotens} yields
  \begin{equation*}
    \mathrm{hom}_{Z(\bG)\bH}(V_1|_{Z(\bG)\bH},W)\cong \mathrm{hom}_\bG\left(V_1,\mathrm{Ind}_{Z(\bG)\bH}^\bG W\right)\ne \{0\},
  \end{equation*}
  meaning that $V_1$ fails to be disjoint from $W$ over $Z(\bG)\bH$ and hence also from
  \begin{equation*}
    V|_{\bH} = W|_{\bH}\text{ over }\bH.
  \end{equation*}
  To conclude, observe that
  \begin{itemize}
  \item \Cref{eq:can} agrees on $V$ and $V_1$ due to the noted non-disjointness
    \begin{equation*}
      \mathrm{hom}_{\bH}(V_1,V)\ne 0;
    \end{equation*}
  \item while the bottom left-hand map $\cat{can}:C(\bG)\to \widehat{Z(\bG)}$ of \Cref{eq:cans} annihilates $V_1$ because the latter has trivial central character;
  \item and hence the top right-hand $\cat{can}$ map in \Cref{eq:cans} must also annihilate $V$. 
  \end{itemize}
  This being the desired conclusion, we are done.
\end{proof}



\addcontentsline{toc}{section}{References}

\Addresses

\end{document}